\PassOptionsToPackage{dvipsnames,svgnames,table}{xcolor}

\documentclass[11pt]{amsart}  

\usepackage[vmargin=1in, hmargin=1.3in]{geometry}  
\usepackage[font=small,format=plain,labelfont=bf,up,textfont=it,up]{caption}
\usepackage{microtype}                  		
\usepackage{graphicx}						
\usepackage{amssymb}
\usepackage{amsmath}
\usepackage{amsthm}
\usepackage{xcolor}
\usepackage{wasysym}
\usepackage{array}
\usepackage{pst-node}
\usepackage{tikz-cd}
\usepackage{tikz}
\usepackage{tikz-qtree}
\usepackage{mathtools}
\usepackage{amsfonts}
\usepackage{bbm}
\usepackage{listings}
\usepackage{MnSymbol}
\usepackage{multirow}
\usepackage{centernot}
\usepackage{ stmaryrd }
\usepackage{subcaption}
\usepackage[backref=page, bookmarks, bookmarksdepth=2, colorlinks=true, linkcolor=blue, citecolor=blue, urlcolor=blue]{hyperref}
\usepackage{enumitem}

\usepackage[export]{adjustbox}
\usepackage{verbatim}
\usepackage{makecell}
\usepackage{xfrac}
\usepackage{float}

\apptocmd{\thebibliography}{\raggedright}{}{}
\renewcommand*{\backref}[1]{}
\renewcommand*{\backrefalt}[4]{%
    \ifcase #1 (Not cited.)%
    \or        (Cited on page~#2.)%
    \else      (Cited on pages~#2.)%
    \fi}

\theoremstyle{plain}
\newtheorem*{theorem*}{Theorem}
\newtheorem*{conjecture*}{Conjecture}
\newtheorem*{corollary*}{Corollary}
\newtheorem{theorem}{Theorem}[section]
\newtheorem{maintheorem}{Theorem}

\newtheorem{proposition}[theorem]{Proposition}
\newtheorem{lemma}[theorem]{Lemma}

\newtheorem{corollary}[theorem]{Corollary}
\newtheorem{conjecture}[theorem]{Conjecture}

\theoremstyle{definition}

\theoremstyle{remark}

\theoremstyle{remark}
\newtheorem{rmk}[theorem]{Remark}
\newenvironment{remark}[1][]{\begin{rmk}[#1]}{\end{rmk}}
\newtheorem{eg}[theorem]{Example}

\definecolor{colorblind_blue}{RGB}{0,114,178}
\definecolor{colorblind_orange}{RGB}{213,94,0}
\definecolor{colorblind_green}{RGB}{0,158,115}
\definecolor{colorblind_purple}{RGB}{204,121,167}
\definecolor{colorblind_darkpurple}{RGB}{126,41,84}

\newcommand{\arxiv}[1]{\href{http://arxiv.org/abs/#1}{{\tt arXiv:#1}}}

\newcommand{\Z}{\mathbb Z}
\newcommand{\Q}{\mathbb Q}

\newcommand\Mod{\ensuremath{\operatorname{Mod}}}

\title[Abelianizations of finite-index subgroups of the
handlebody group]{Abelianizations of finite-index subgroups of the handlebody group}

\author{Annie Holden}
\address{Department of Mathematics, University of Pennsylvania, Philadelphia, PA, USA}
\email{holdena@sas.upenn.edu}

\begin{document}

\newpage

\begin{abstract}

For genus $\geq 4$, it is an open question whether the mapping class group of a handlebody contains a finite-index subgroup with nontrivial rational abelianization. In this paper, we provide evidence that no such subgroup exists. First, we prove that, for all such finite-index subgroups $\Gamma$, meridian multitwists vanish in $H_1(\Gamma; \Q)$. Next, we show that $H_1(\Gamma; \Q) = 0$ for finite-index subgroups $\Gamma$ containing the handlebody Torelli group, or large enough subgroups of the twist group or the handlebody Johnson kernel.

\end{abstract}

\maketitle
\thispagestyle{empty}

\section{Introduction}

Let $\Sigma_g$ be an oriented surface of genus $g$. The mapping class group $\Mod(\Sigma_g)$ is the group of isotopy classes of orientation-preserving diffeomorphisms of $\Sigma_g$. A famous question of Ivanov asks whether $\Mod(\Sigma_g)$ admits a finite-index subgroup with infinite abelianization. See \cite{Ivanov} for an overview and see Problem 2.11 (A) in Kirby’s list \cite{Kirby}. This question is open for $g \geq 3$. In contrast, the analogue of Ivanov's question for the outer automorphism group of a free group, $\operatorname{Out}(F_g)$, has been completely determined. A deep theorem of Kaluba--Kielak--Nowak ($g \geq 6$) \cite{PropertyT6}, Kaluba--Nowak--Ozawa ($g=5$) \cite{PropertyT5}, and Nitsche ($g=4$) \cite{Nitsche} says that $\operatorname{Out}(F_g)$ has Kazhdan's Property (T) when $g \geq 4$. This implies that all finite-index subgroups have finite abelianization.

Fix a handlebody $\mathcal{V}_g$ with $\partial \mathcal{V}_g = \Sigma_g$. The handlebody group $\mathcal{H}_g$ is the subgroup of $\Mod(\Sigma_g)$ consisting of mapping classes that extend to $\mathcal{V}_g$. The handlebody group acts on $\pi_1(\mathcal{V}_g)$, inducing a surjection from
$\mathcal{H}_g$ onto $\operatorname{Out}(F_g)$. In this way, $\mathcal{H}_g$ is a natural bridge between $\operatorname{Mod}(\Sigma_g)$ and $\operatorname{Out}(F_g)$. With this motivation, Hensel \cite[Question 8.7]{Hensel} posed the following ``likely very hard" question: for $g \geq 4$, does $\mathcal{H}_g$ contain a finite-index subgroup with infinite abelianization? Equivalently, for $g \geq 4$, does $\mathcal{H}_g$ contain a finite-index subgroup $\Gamma$ such that $H_1(\Gamma; \Q) \neq 0$?

\begin{remark}
    The groups $\operatorname{Out}(F_2)$ and $\operatorname{Out}(F_3)$ have finite-index subgroups with infinite abelianizations, and so the same is true for $\mathcal{H}_2$ and $\mathcal{H}_3$. We therefore restrict our attention to the cases of $g \geq 4$.
\end{remark}

In this paper, we present evidence that $\mathcal{H}_g$ does not admit such a finite-index subgroup. Since our approach is motivated by analogous results for $\operatorname{Mod}(\Sigma_g)$, we review what is known about $\operatorname{Mod}(\Sigma_g)$ to place our work in context.

\subsection{Vanishing mapping classes} Putman \cite[Theorem A]{Putman} and Bridson \cite[Remark 1]{Bridson} independently proved that, if $g \geq 3$ and $\Gamma$ is a finite-index subgroup of $\operatorname{Mod}(\Sigma_g)$, then powers of Dehn twists vanish in $H_1(\Gamma;\Q)$. This was easily extended by Putman--Wieland \cite[Corollary 2.10]{PutmanWieland} to multitwists.\footnote{A multitwist is a product of powers of Dehn twists $T_{\gamma_1}^{k_1} \dots T_{\gamma_m}^{k_m}$, where $\gamma_{1}, \dots, \gamma_{m}$ are disjoint simple closed curves on $\Sigma_g$.} 

A \emph{meridian} is a simple closed curve on $\Sigma_g$ that bounds a disk in $\mathcal{V}_g$, and a \emph{meridian twist} is a Dehn twist about a meridian. If a multitwist $M = T_{\gamma_1}^{k_1} \dots T_{\gamma_m}^{k_m}$ lies in $\mathcal{H}_g$, the curves $\gamma_i$ are necessarily meridians and we say $M$ is a \emph{meridian multitwist}. We have the following theorem.

\begin{maintheorem}[Meridian multitwists vanish] \label{maintheorem:multitwistsdie}
    For $g \geq 4$, let $\Gamma$ be a finite-index subgroup of $\mathcal{H}_g$. If $M$ is a meridian multitwist in $\Gamma$, then the class $[M] = 0$ in $H_1(\Gamma; \Q)$.
\end{maintheorem}

Next, we prove Theorems \ref{maintheorem:B}, \ref{maintheorem:torelli}, and \ref{maintheorem:johnsonkernel}, which each give conditions on finite-index subgroups of $\mathcal{H}_g$ that imply these subgroups have finite abelianizations. 

\subsection{Subgroups containing large pieces of the twist group} If $\Gamma$ is a subgroup of $\mathcal{H}_g$, we define $T(\Gamma)$ to be the group generated by the set $$\{T_{\gamma}^n \in \Gamma \; | \; \gamma \text{ is a meridian} \}.$$ The twist group $\mathcal{T}_g$ is the subgroup of $\mathcal{H}_g$ generated by meridian twists, and is the kernel of the action on $\pi_1(\mathcal{V}_g)$. See Section \ref{section:actionpi1}. If $\Gamma$ contains $\mathcal{T}_g$, then $T(\Gamma) = \Gamma \cap \mathcal{T}_g$, but the converse does not hold. We have the following theorem.

\begin{maintheorem}[Subgroups containing large pieces of the twist group] \label{maintheorem:B}
    For $g \geq 4$, let $\Gamma$ be a finite-index subgroup of $\mathcal{H}_g$ for which $[\Gamma \cap \mathcal{T}_g : T(\Gamma)] < \infty$. Then $H_1(\Gamma; \Q) = 0$.
\end{maintheorem}

Theorem \ref{maintheorem:B} has the following immediate consequences, which are weaker but whose conditions are easier to state and check. 

\begin{corollary*}[Corollary \ref{corollary:kernelaction} and Corollary \ref{corollary:twist}] For $g \geq 4$, let $\Gamma$ be a finite-index subgroup of $\mathcal{H}_g$. We have $H_1(\Gamma; \Q) = 0$ in either of the following cases: \begin{itemize}
    \item The kernel of the action of $\Gamma$ on $\pi_1(\mathcal{V}_g)$ is generated by meridian multitwists\footnote{Equivalently, $\Gamma \cap \mathcal{T}_g$ is generated by meridian multitwists.}; or
    \item The group $\Gamma$ contains the twist group $\mathcal{T}_g$.
\end{itemize}
\end{corollary*}

Since the action of $\operatorname{Mod}(\Sigma_g)$ on $\pi_1(\Sigma_g)$ is faithful, there is no analogue of Theorem \ref{maintheorem:B} for $\operatorname{Mod}(\Sigma_g)$. Nonetheless, our argument is inspired by the proof of \cite[Theorem B]{Putman}. We also use the fact that the image of $\Gamma$ in $\operatorname{Out}(F_g)$ has finite abelianization.

\subsection{Subgroups containing the handlebody Torelli group} \label{section:introTorelli} The group $\operatorname{Mod}(\Sigma_g)$ acts on $H_1(\Sigma_g; \Z)$ and the kernel of this action is the \emph{Torelli group}, denoted $\mathcal{I}_g$. Hain \cite{Hain} and McCarthy \cite{McCarthy} showed that any finite-index subgroup $\Gamma$ of $\operatorname{Mod}(\Sigma_g)$ containing the Torelli group satisfies $H_1(\Gamma; \Q) = 0$. The \emph{handlebody Torelli group} is the intersection $\mathcal{HI}_g := \mathcal{H}_g \cap \mathcal{I}_g$. Equivalently, $\mathcal{HI}_g$ is the subgroup of $\mathcal{H}_g$ that acts trivially on $H_1(\Sigma_g; \Z)$. We have the following theorem.

\begin{maintheorem}[Subgroups containing the handlebody Torelli group] \label{maintheorem:torelli}
     For $g \geq 3$, let $\Gamma$ be a finite-index subgroup of $\mathcal{H}_g$ that contains the handlebody Torelli group $\mathcal{HI}_g$. Then $H_1(\Gamma; \Q) = 0$.
\end{maintheorem} 

\begin{remark}
    The handlebody group also acts on $H_1(\mathcal{V}_g; \Z)$ and the kernel of this action is another Torelli subgroup, denoted $\mathcal{H}_B \mathcal{I}_g$. Since $\mathcal{H}_B \mathcal{I}_g$ contains both $\mathcal{T}_g$ and $\mathcal{HI}_g$, Theorem \ref{maintheorem:B} (for $g \geq 4$) and Theorem \ref{maintheorem:torelli} (for $g \geq 3$) independently tell us that any finite-index subgroup containing $\mathcal{H}_B \mathcal{I}_g$ has finite abelianization.
\end{remark}

\subsection{Subgroups containing large pieces of the handlebody Johnson kernel} Of infinite index in the Torelli group is the \emph{Johnson kernel} $\mathcal{K}_g$, the subgroup generated by Dehn twists about separating curves. Putman \cite[Theorem B]{Putman} strengthened the results of Hain and McCarthy by showing that any finite-index subgroup $\Gamma$ of $\operatorname{Mod}(\Sigma_g)$ containing a large piece of $\mathcal{K}_g$ has finite abelianization. Specifically, he showed that if the group generated by the set $\{ T_{\gamma}^n \in \Gamma \; | \; \gamma \text{ is a separating curve} \}$ is of finite index in $\Gamma \cap \mathcal{K}_g$, then $H_1(\Gamma; \Q) = 0$. 

Consider the \emph{handlebody Johnson kernel} $\mathcal{HK}_g := \mathcal{H}_g \cap \mathcal{K}_g$. If $\Gamma$ is a subgroup of $\mathcal{H}_g$, define $K(\Gamma)$ to be the subgroup of $\Gamma \cap \mathcal{HK}_g$ generated by the set $$\{ T_{\gamma}^n \in \Gamma \; | \; \gamma \text{ is a separating meridian} \}.$$ We have the following theorem.

\begin{maintheorem}[Subgroups containing large pieces of the handlebody Johnson kernel] \label{maintheorem:johnsonkernel}
    For $g \geq 4$, let $\Gamma$ be a finite-index subgroup of $\mathcal{H}_g$ for which $[\Gamma \cap \mathcal{HK}_g : K(\Gamma)] < \infty$. Then $H_1(\Gamma; \Q) = 0$.
\end{maintheorem}

It is natural to conjecture the following.

\begin{conjecture*}[Conjecture \ref{conjecture:rationalabelHI}] 
    The handlebody Johnson kernel $\mathcal{HK}_g$ is generated by separating meridian twists.
\end{conjecture*}

If Conjecture \ref{conjecture:rationalabelHI} is true and $\Gamma$ contains $\mathcal{HK}_g$, then $K(\Gamma) = \mathcal{HK}_g$. In this case, we get the following corollary which strengthens Theorem \ref{maintheorem:torelli}.

\begin{corollary*}[Corollary \ref{corollary:johnsonkernel}]
     Assume Conjecture \ref{conjecture:rationalabelHI}. For $g \geq 4$, let $\Gamma$ be a finite-index subgroup of $\mathcal{H}_g$ that contains the handlebody Johnson kernel $\mathcal{HK}_g$. Then $H_1(\Gamma; \Q) = 0$.
\end{corollary*}

\subsection{Future work} There are results for $\operatorname{Mod}(\Sigma_g)$ for which we do not have analogues for $\mathcal{H}_g$. See, in particular, the work of Ershov--He \cite[Theorem 1.9]{ErshovHe} and Putman--Wieland \cite{PutmanWieland}. In future work, it would be interesting to study $\mathcal{H}_g$ from these points of view.

\subsection{Acknowledgments} I would like to thank Filippo Bianchi for introducing me to this problem and for many discussions that greatly influenced this project. I thank Andrew Putman for guidance and helpful conversations, especially regarding the proof of Lemma \ref{lemma:BPMvanish}. I also thank Akash Narayanan, Audriana Houtz, and Kenz Kallal for helpful discussions.

\section{Handlebody groups} A genus $g$ handlebody $\mathcal{V}_g$ is an oriented 3-manifold with $\partial \mathcal{V}_g = \Sigma_g$ obtained from a 3-ball by attaching $g$ one-handles. We denote by $\mathcal{V}_g^b$ a handlebody with $b$ embedded disks (``spots") in its boundary. We slightly abuse notation and define the boundary of the spotted handlebody $\mathcal{V}_g^b$ to be a surface with $b$ boundary components, $\Sigma_g^b$. Then $\Mod(\Sigma_g^b)$ is the group of isotopy classes of diffeomorphisms of $\Sigma_g^b$ that fix the boundary components pointwise, and the handlebody group $\mathcal{H}_g^b$ is the subgroup of $\Mod(\Sigma_g^b)$ consisting of mapping classes that extend to $\mathcal{V}_g^b$.

\subsection{Birman exact sequence} \label{section:birman} The groups $\mathcal{H}_g^b$ and $\mathcal{H}_g^{b+1}$ are related in the following way. Fix a boundary component of $\Sigma_g^{b+1}$. There exists a ``forgetful map" $\mathcal{H}_g^{b+1} \rightarrow \mathcal{H}_g^b$ given by gluing a disk to this boundary component. Letting $U \Sigma_g^b$ be the unit tangent bundle of $\Sigma_g^b$, there is a subgroup of $\operatorname{Mod}(\Sigma_g^{b+1})$ isomorphic to $\pi_1(U \Sigma_g^b)$ that consists of maps that ``push the boundary component around curves." This subgroup lies in $\mathcal{H}_g^{b+1}$ (see \cite[Section 3]{Hensel}) and is exactly the kernel of the forgetful map. This is summarized in the Birman exact sequence: $$1 \rightarrow \pi_1(U \Sigma_g^b) \rightarrow \mathcal{H}_g^{b+1} \rightarrow \mathcal{H}_g^b \rightarrow 1.$$

\subsection{Multitwists in the handlebody group} There are many straightforward tests for determining whether a mapping class belongs to $\mathcal{H}_g^b$. See \cite[Corollary 5.11]{Hensel}. For this paper, the following proposition will suffice.

\begin{proposition}[{\cite[Theorem 1.11]{Oertel}}, {\cite[Theorem 1]{McCullough}}] \label{prop:multitwistmembership}
    Suppose $\gamma_1, \dots, \gamma_k$ are disjoint simple closed curves on $\Sigma_g^b$. Then the multitwist $$T_{\gamma_1}^{n_1} \dots T_{\gamma_k}^{n_k}$$ lies in $\mathcal{H}_g^b$ if and only if, up to reordering, we have: \begin{itemize}
        \item The curves $\gamma_1, \dots, \gamma_r$ each bound embedded disks in $\mathcal{V}_g^b$; and
        \item For all $i = r+2, r+4, \dots, r+2\ell$ where $r+2\ell = k$, the union $\gamma_{i-1} \cup \gamma_i$ bounds an annulus embedded in $\mathcal{V}_g^b$, and $n_{i-1} = -n_i$.
    \end{itemize}
\end{proposition}

Proposition \ref{prop:multitwistmembership} gives us the following two important examples of elements in $\mathcal{H}_g^b$. When $\gamma$ bounds a disk in the handlebody, we call $\gamma$ a \emph{meridian} and we call $T_{\gamma}$ a \emph{meridian twist}. When the curves $\gamma_1$ and $\gamma_2$ bound an annulus in the handlebody, we call the product $T_{\gamma_1} T_{\gamma_2}^{-1}$ an \emph{annulus twist}.

\subsection{Actions on fundamental groups} \label{section:actionpi1}
The handlebody group $\mathcal{H}_g$ acts on $\pi_1(\mathcal{V}_g) \cong F_g$ via outer automorphisms, giving a map to $\operatorname{Out}(F_g)$. Griffiths \cite{Griffiths} proves that this map is a surjection. Luft \cite{Luft} proves that the kernel is the twist group $\mathcal{T}_g$: the subgroup of $\mathcal{H}_g$ generated by meridian twists. This is summarized in the following short exact sequence: $$1 \rightarrow \mathcal{T}_g \rightarrow \mathcal{H}_g \rightarrow \operatorname{Out}(F_g) \rightarrow 1.$$

\section{Group homology}

We are interested in the group homology of handlebody groups and their subgroups. A good source on group homology is \cite{Brown}, and an important result in group homology is the five-term exact sequence, which we introduce in the following subsection.

\subsection{Five-term exact sequence} Consider a short exact sequence of groups $$1 \rightarrow K \rightarrow G \rightarrow Q \rightarrow 1.$$ The group $G$ acts on its normal subgroup $K$ by conjugation, inducing an action on $H_1(K; R)$ for any ring $R$. Since this action is trivial when restricted to $K$, it descends to an action of $G / K \cong Q$ on $H_1(K;R)$. We define the coinvariants of this action, denoted $H_1(K;R)_Q$, to be the quotient of $H_1(K;R)$ by the subgroup spanned by $\{ k - q(k) \; | \; k \in H_1(K;R), \; q \in Q \}$. Equivalently, this is the largest quotient of $H_1(K; R)$ on which $Q$ acts trivially. Associated to our short exact sequence is the five-term exact sequence: $$H_2(G; R) \rightarrow H_2(Q;R) \rightarrow H_1(K;R)_Q \rightarrow H_1(G;R) \rightarrow H_1(Q;R) \rightarrow 0.$$ See \cite[Corollary VII.6.4]{Brown}.

\subsection{Abelianization of the handlebody group} In the following proposition, we address the abelianization of the handlebody group with at most two embedded disks.

\begin{proposition} \label{proposition:H1handlebody}
     For all $b \in \{0,1,2 \}$, we have: \[ H_1(\mathcal{H}_g^b; \Z) \cong \begin{cases}
     \Z \oplus \sfrac{\Z}{2\Z} & g=1 \\
     \Z \oplus \sfrac{\Z}{2\Z} \oplus \sfrac{\Z}{2\Z} & g=2 \\
     \sfrac{\Z}{2\Z} & g \geq 3 \\
     \end{cases} .
     \]
\end{proposition}

In particular, $H_1(\mathcal{H}_g^b; \Q) = 0$ when $g \geq 3$ and $b \leq 2$. This result is already known for $b \in \{0,1\}$. See \cite[Lemma 2.4]{Ishida} and \cite[Remark 2.5]{Ishida}. For future use, we extend this result to mapping class groups of handlebodies with two embedded disks. 

\begin{proof}[Proof of Proposition \ref{proposition:H1handlebody}] Since the cases of $b \in \{0,1\}$ are known, it suffices to show the isomorphism $H_1(\mathcal{H}_g^2; \Z) \cong H_1(\mathcal{H}_g^1; \Z)$. The five-term exact sequence of the extension $$1 \rightarrow \pi_1(U\Sigma_g^1) \rightarrow \mathcal{H}_g^2 \rightarrow \mathcal{H}_g^1 \rightarrow 1$$ contains the exact sequence\footnote{We use the fact that $U \Sigma_g^1$ is a $K(\pi_1(U \Sigma_g^1),1)$.} $$H_1( U \Sigma_g^1; \Z)_{\mathcal{H}_g^1} \rightarrow H_1(\mathcal{H}_g^2; \Z) \rightarrow H_1(\mathcal{H}_g^1; \Z) \rightarrow 0.$$ We show that $H_1(U \Sigma_g^1; \Z)_{\mathcal H_g^1}=0$. Our argument follows that of \cite[Lemma 2.1]{Ishida}.

\begin{figure}[htbp]
    \centering
    \includegraphics[scale=.5]{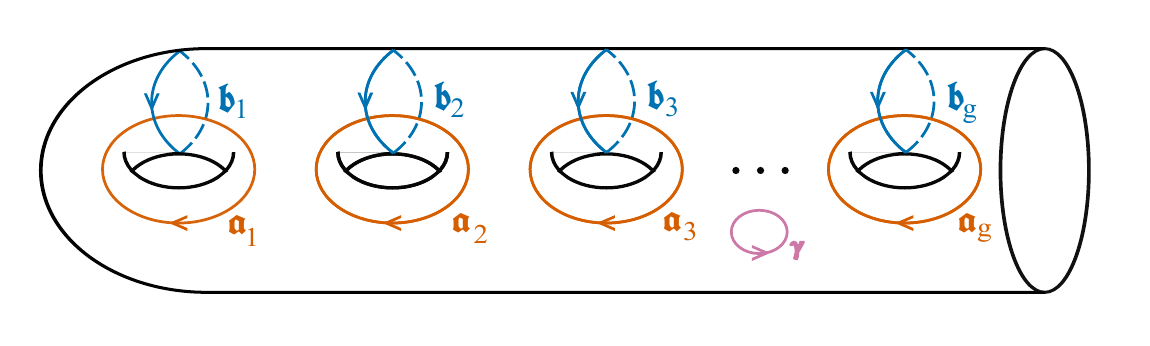}
    \caption{A basis for $H_1(\Sigma_g^1; \Z)$}
    \label{figure:homologybasis}
\end{figure}

A basis for $H_1(\Sigma_g^1;\Z)$ is $\{ a_1, \dots, a_g, b_1, \dots, b_g \}$, where Figure \ref{figure:homologybasis} shows curves such that $[\mathfrak{a}_i] = a_i$ and $[\mathfrak{b}_i] = b_i$. Figure \ref{figure:homologybasis} also shows a nullhomologous curve $\gamma$. The homology classes $a_i$ and $b_i$ lift to $\tilde{a}_i$ and $\tilde{b}_i$, respectively, in $H_1(U \Sigma_g^1; \Z)$, and $[\gamma]$ lifts to $z$. This gives us $$H_1(U \Sigma_g^1; \Z) = \Z \langle \tilde{a}_1, \dots, \tilde{a}_g, \tilde{b}_1, \dots, \tilde{b}_g, z \rangle. $$

We show that each of these generators vanishes in the quotient $H_1(U \Sigma_g^1; \Z)_{\mathcal{H}_g^1}$. First, we show that the classes of $\tilde{b}_1, \dots, \tilde{b}_g$ vanish in $H_1(U \Sigma_g^1; \Z)_{\mathcal{H}_g^1}$. For each $1 \leq i \leq g$, we have $$T_{\mathfrak{b}_i} (\tilde{a}_i) = \tilde{a}_i + \tilde{b}_i.$$ It follows that $$\tilde{a}_i - T_{\mathfrak{b}_i}(\tilde{a}_i) = -\tilde{b}_i$$ vanishes in the quotient.

\begin{figure}[htbp]
    \centering
    \includegraphics[scale=.5]{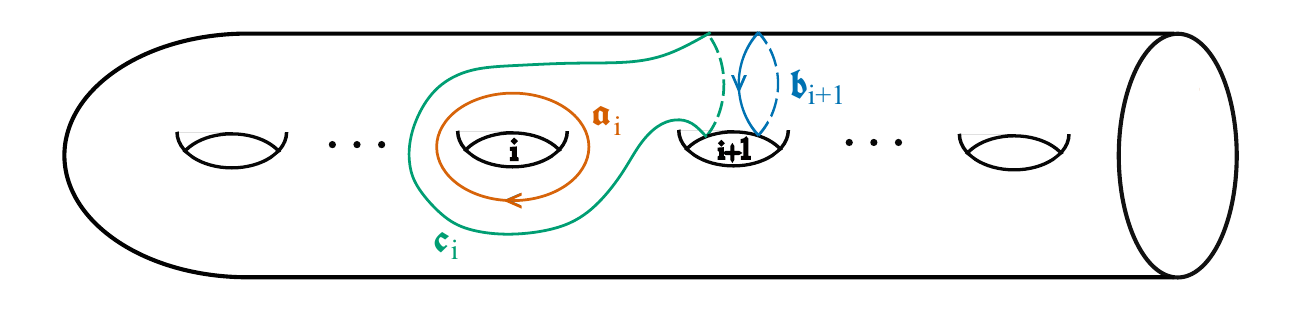}
    \caption{The curves involved in the map $h_i$}
    \label{figure:curvesforhimap}
\end{figure}
    
Now consider the curves shown in Figure \ref{figure:curvesforhimap}. For $i =1, \dots, g-1$, the mapping class $$h_i := T_{\mathfrak{c}_i} T_{\mathfrak{a}_i}^{-1} T_{\mathfrak{b}_{i+1}}^{-1}$$ is the product of an annulus twist and a disk twist, and so lies in $\mathcal{H}_g^1$ by Proposition \ref{prop:multitwistmembership}. We will use this map to show that the class of $z$ vanishes in $H_1(U \Sigma_g^1; \Z)_{\mathcal{H}_g^1}$. Applying left-handed Dehn twists, we have $$h_1(\tilde{b}_1) = \tilde{b}_1 - \tilde{b}_{2} + z,$$ and so $$\tilde{b}_1 - h_1(\tilde{b}_1) = \tilde{b}_{2} - z.$$ It follows that the class of $z$ is identified with the class of $\tilde{b}_2$ in the quotient, and so it vanishes.

Next, we show that the classes of $\tilde{a}_1, \dots, \tilde{a}_{g-1}$ vanish in $H_1(U \Sigma_g^1; \Z)_{\mathcal{H}_g^1}$. For $1 \leq i \leq g-1$, we have $$h_i(\tilde{a}_{i+1}) = \tilde{a}_i + \tilde{a}_{i+1} - z,$$ and so $$\tilde{a}_{i+1} - h_i(\tilde{a}_{i+1}) = -\tilde{a}_i + z.$$ It follows that the class of $\tilde{a}_i$ is identified with the class of $z$ in the quotient, and so it vanishes. 
    
It remains to show that the class of $\tilde{a}_g$ vanishes in $H_1(U \Sigma_g^1; \Z)_{\mathcal{H}_g^1}$. The symplectic automorphism that swaps $a_1 \leftrightarrow a_g$ and $b_1 \leftrightarrow b_g$ can be realized by an element $f \in \mathcal{H}_g^1$. See, for example, the element that swaps the $1^{\text{st}}$ and $g^{\text{th}}$ handles of the handlebody in \cite{Suzuki}. Acting on $H_1(U \Sigma_g^1; \Z)_{\mathcal{H}_g^1}$, we have $$f(\tilde{a}_1) = \tilde{a}_g + nz$$ for some $n \in \Z$. Then $\tilde{a}_g$ is identified with $\tilde{a}_1 -nz$ in the quotient, and so it vanishes. \end{proof}

\subsection{Finite-index subgroups} We have the following lemma. 

\begin{lemma} \label{lemma:finiteindexhomology}
    Let $G_2$ be a finite-index subgroup of $G_1$. Then, for all $k$, the map $$\iota_*: H_k(G_2; \Q) \rightarrow H_k(G_1; \Q)$$ induced by the inclusion is surjective.
\end{lemma}

\begin{proof}
    When $G_2$ is a finite-index subgroup of $G_1$, there exists a transfer map $$t: H_k(G_1; \Q) \rightarrow H_k(G_2; \Q)$$ with the property that $\iota_* \circ t$ is multiplication by $[G_1: G_2]$. See \cite[Proposition 9.5 (iii)]{Brown}. Thus, with coefficients in $\Q$, multiplication by $\frac{1}{[G_1:G_2]}$ is an inverse to $\iota_* \circ t$. It follows that $\iota_*$ is surjective. \end{proof}

\section{Proof of Theorem A}

In this section, we prove Theorem \ref{maintheorem:multitwistsdie}. We start with the following proposition, whose statement and proof are inspired by \cite[Theorem A]{Putman}.

\begin{proposition}[Powers of meridian twists vanish] \label{proposition:twistsvanish} For $g \geq 4$, let $\Gamma$ be a finite-index subgroup of $\mathcal{H}_g$. Let $\gamma$ be a meridian on $\Sigma_g$ and pick $n \geq 1$ such that $T_{\gamma}^n \in \Gamma$. Then the class $[T_{\gamma}^n] = 0$ in $H_1(\Gamma; \Q)$. \end{proposition} 

\begin{remark} \label{remark:powersmeridiantwists}
    It follows from Proposition \ref{prop:multitwistmembership} that $T_{\gamma}^m$ belongs to $\mathcal{H}_g$ if and only if $\gamma$ is a meridian. Consequently, if $\Gamma$ is a finite-index subgroup of $\mathcal{H}_g$, then there exists $n \geq 1$ such that $T_{\gamma}^n \in \Gamma$ if and only if $\gamma$ is a meridian.
\end{remark}

To prove Proposition \ref{proposition:twistsvanish}, we introduce the following lemma.

\begin{lemma} \label{lemma:subsurface}
    Let $\gamma$ be a meridian on $\Sigma_g$ for $g \geq 4$. Then there exists a subsurface $S \hookrightarrow \Sigma_g$ with the following properties: \begin{itemize}
        \item We have $S \cong \Sigma_{g_S}^{b_S}$ with $b_S \in \{1,2\}$ and $g_S \geq 3$; and
        \item Letting $i: \operatorname{Mod}(S) \rightarrow \operatorname{Mod}(\Sigma_g)$ be the map induced by extending by the identity, there exists $\beta \subset \partial S$ such that $i(T_{\beta}) = T_{\gamma}$; and
        \item The group $\mathcal{H}(S) := i(\operatorname{Mod}(S)) \cap \mathcal{H}_g$ is isomorphic to $\mathcal{H}_{g_S}^{b_S}$.
    \end{itemize}
\end{lemma}

\begin{proof}
    There are two cases. If $\gamma$ is nonseparating, let $S$ be the complement of a regular neighborhood of $\gamma$. Then $S \cong \Sigma_{g-1}^2$ and $g-1 \geq 3$. If $\gamma$ is separating, cutting $\Sigma_g$ along $\gamma$ produces two surfaces with boundary. Let $S$ be the component of maximal genus. In this case, $S \cong \Sigma_{g_S}^1$ for $g_S \geq 3$. Cutting a handlebody along a meridian yields a spotted handlebody. In both cases, $S$ bounds a genus $g_S$ subhandlebody of $\mathcal{V}_g$ with $b_S$ spots. Its handlebody group $\mathcal{H}_{g_S}^{b_S}$ is exactly $\mathcal{H}(S)$.
    \end{proof}

This sets us up to prove Proposition \ref{proposition:twistsvanish}.

\begin{proof}[Proof of Proposition \ref{proposition:twistsvanish}] As in Lemma \ref{lemma:subsurface}, let $S$ be a subsurface of $\Sigma_g$ and let $$i: \operatorname{Mod}(S) \rightarrow \operatorname{Mod}(\Sigma_g)$$ be the map induced by extending by the identity. Let $\mathcal{H}(S)$ be the intersection $\mathcal{H}_g \cap i(\operatorname{Mod}(S))$. Then $\mathcal{H}(S) \cong \mathcal{H}_{g_S}^{b_S}$ for $g_S \geq 3$ and $b_S \leq 2$. There is an inclusion $$\iota : \mathcal{H}(S) \rightarrow \mathcal{H}_g$$ and we define $\Gamma_S$ to be $\iota^{-1}(\Gamma)$. It is enough to show that $[T_{\beta}^n]_{\Gamma_S} = 0$. (If a power of $T_{\beta}^n$ lies in $[\Gamma_S, \Gamma_S]$, then the same power of $\iota (T_{\beta}^n) = T_{\gamma}^n$ lies in $[\Gamma, \Gamma]$.) Let $\overline{S}$ be the surface obtained by gluing a disk to $\beta$, and let $\Gamma_{\overline{S}}$ be the image of $\Gamma_S$ in $\mathcal{H}(\overline{S})$ under the forgetful map from Section \ref{section:birman}. We have the following commutative diagram. \[\begin{tikzcd}[ampersand replacement=\&]
	1 \& {\mathbb{Z}} \& {\Gamma_S} \& {\Gamma_{\overline{S}}} \& 1 \\
	1 \& {\mathbb{Z}} \& {\mathcal{H}(S)} \& {\mathcal{H}(\overline{S})} \& 1
	\arrow[from=1-1, to=1-2]
	\arrow[from=1-2, to=1-3]
	\arrow["{\times n}", from=1-2, to=2-2]
	\arrow["Forget", from=1-3, to=1-4]
	\arrow[from=1-3, to=2-3]
	\arrow[from=1-4, to=1-5]
	\arrow[from=1-4, to=2-4]
	\arrow[from=2-1, to=2-2]
	\arrow[from=2-2, to=2-3]
	\arrow["Forget", from=2-3, to=2-4]
	\arrow[from=2-4, to=2-5]
\end{tikzcd}\]

The groups $\Z$ are generated by $T_{\beta}^n$ and $T_{\beta}$, respectively. This induces the following five-term exact sequences.

\[\begin{tikzcd}[ampersand replacement=\&]
	{H_2(\Gamma_{\overline{S}}; \Q)} \& {\mathbb{Q}} \& {H_1(\Gamma_S; \Q)} \& {H_1(\Gamma_{\overline{S}} ; \Q)} \& 0 \\
	{H_2(\mathcal{H}(\overline{S}) ; \Q)} \& {\mathbb{Q}} \& {H_1(\mathcal{H}(S) ; \Q)} \& {H_1(\mathcal{H}(\overline{S}) ; \Q)} \& 0
	\arrow["{f_1}", from=1-1, to=1-2]
	\arrow["{f_2}"', from=1-1, to=2-1]
	\arrow[from=1-2, to=1-3]
	\arrow["\cong", from=1-2, to=2-2]
	\arrow["{f_4}", from=1-3, to=1-4]
	\arrow[from=1-3, to=2-3]
	\arrow[from=1-4, to=1-5]
	\arrow[from=1-4, to=2-4]
	\arrow["{f_3}", from=2-1, to=2-2]
	\arrow[from=2-2, to=2-3]
	\arrow[from=2-3, to=2-4]
	\arrow[from=2-4, to=2-5]
\end{tikzcd}\]

The proof is complete once we show that $f_4$ is an isomorphism. By Lemma \ref{lemma:finiteindexhomology}, $f_2$ is surjective. By Proposition \ref{proposition:H1handlebody}, we have $H_1(\mathcal{H}(S); \Q) = 0$, so $f_3$ is a surjection. By commutativity of the diagram, $f_1$ is surjective, and hence $f_4$ is an isomorphism. \end{proof}

Theorem \ref{maintheorem:multitwistsdie}, which we restate, is an easy extension of Proposition \ref{proposition:twistsvanish}. Our proof follows that of \cite[Corollary 2.10]{PutmanWieland}. 

\begin{theorem*}[Theorem \ref{maintheorem:multitwistsdie}] For $g \geq 4$, let $\Gamma$ be a finite-index subgroup of $\mathcal{H}_g$. If $M$ is a meridian multitwist in $\Gamma$, then the class $[M] = 0$ in $H_1(\Gamma; \Q)$.
\end{theorem*}

\begin{proof}
    Let $M = T_{\gamma_1}^{k_1} \dots T_{\gamma_m}^{k_m}$ be a meridian multitwist. For each $1 \leq i \leq m$, there exists $K_i \geq 1$ such that $T_{\gamma_i}^{K_i k_i} \in \Gamma$. Let $K = \max\{ K_1, \dots, K_m \}$. Since the $T_{\gamma_i}$ commute, we have $$M^K = T_{\gamma_1}^{K k_1} \dots T_{\gamma_m}^{K k_m}.$$ Proposition \ref{proposition:twistsvanish} says that the class $[T_{\gamma_i}^{K k_i}]$ vanishes in $H_1(\Gamma; \Q)$ for $1 \leq i \leq m$, so the class $[M^K]$ (and thus $[M]$) vanishes as well.
\end{proof}

\section{Proof of Theorem B} \label{section:prooftheorem:kernelaction}

In this section, we prove Theorem \ref{maintheorem:B}. Recall that, if $\Gamma$ is a subgroup of $\mathcal{H}_g$, we define $T(\Gamma)$ to be the group generated by the set $$\{T_{\gamma}^n \in \Gamma \; | \; \gamma \text{ is a meridian} \}.$$ The twist group $\mathcal{T}_g$ (see Section \ref{section:actionpi1}) is the kernel of the action of $\mathcal{H}_g$ on $\pi_1(\mathcal{V}_g)$ and is generated by the set of all meridian twists. We have the following theorem.

\begin{theorem*}[Theorem \ref{maintheorem:B}]
    For $g \geq 4$, let $\Gamma$ be a finite-index subgroup of $\mathcal{H}_g$ for which $[\Gamma \cap \mathcal{T}_g : T(\Gamma)] < \infty$. Then $H_1(\Gamma; \Q) = 0$.
\end{theorem*}

\begin{proof}
    The five-term exact sequence of the extension $$1 \rightarrow \Gamma \cap \mathcal{T}_g \rightarrow \Gamma \rightarrow \Gamma / (\Gamma \cap \mathcal{T}_g) \rightarrow 1$$ contains the exact sequence $$H_1(\Gamma \cap \mathcal{T}_g; \Q)_{\Gamma / (\Gamma \cap \mathcal{T}_g)} \xrightarrow{i} H_1(\Gamma; \Q) \rightarrow H_1(\Gamma / (\Gamma \cap \mathcal{T}_g); \Q) \rightarrow 0.$$ Since $\operatorname{Out}(F_g)$ has property (T) and $\Gamma / (\Gamma \cap \mathcal{T}_g)$ is a finite-index subgroup of $\mathcal{H}_g / \mathcal{T}_g \cong \operatorname{Out}(F_g)$, we have $$H_1(\Gamma / (\Gamma \cap \mathcal{T}_g); \Q) = 0.$$ We assume $[\Gamma \cap \mathcal{T}_g : T(\Gamma)] < \infty$, so by Lemma \ref{lemma:finiteindexhomology} the map $$H_1(T(\Gamma); \Q) \rightarrow H_1(\Gamma \cap \mathcal{T}_g; \Q)$$ is surjective. The group $T(\Gamma)$ is generated by meridian twists, which by Theorem \ref{maintheorem:multitwistsdie} vanish in $H_1(\Gamma; \Q)$. It follows that $i$ is the zero map, and our result follows.
\end{proof}

We have the following corollary. 

\begin{corollary} \label{corollary:twist}
     For $g \geq 4$, let $\Gamma$ be a finite-index subgroup of $\mathcal{H}_g$ that contains the twist group $\mathcal{T}_g$. Then $H_1(\Gamma; \Q) = 0$.
\end{corollary}

We can also form the following stronger corollary.  

\begin{corollary} \label{corollary:kernelaction}
    For $g \geq 4$, let $\Gamma$ be a finite-index subgroup of $\mathcal{H}_g$ such that the kernel of the action on $\pi_1(\mathcal{V}_g)$ is generated by meridian multitwists. Then $H_1(\Gamma; \Q) = 0$.
\end{corollary}

\begin{proof} It suffices to show that, if the kernel of the action on $\pi_1(\mathcal{V}_g)$ is generated by meridian multitwists, then the conditions of Theorem \ref{maintheorem:B} are satisfied. It follows from Section \ref{section:actionpi1} that the kernel of the action on $\pi_1(\mathcal{V}_g)$ is exactly $\Gamma \cap \mathcal{T}_g$. Take some meridian multitwist $M = T_{\gamma_1}^{k_1} \dots T_{\gamma_m}^{k_m}$. Then, for each $i = 1, \dots, m$, we must have $T_{\gamma_i}^{k_i \cdot N_i} 
\in \Gamma$ for some $N_i \in \Z$. See Remark \ref{remark:powersmeridiantwists}. Letting $N := \max \{N_1, \dots, N_m\}$, we have $M^N \in T(\Gamma)$, and $T(\Gamma)$ is of finite index in $\Gamma \cap \mathcal{T}_g$.
\end{proof}

In Appendix \ref{appendixKernelAction}, we give an alternate proof of Corollary \ref{corollary:kernelaction} which does not invoke Theorem \ref{maintheorem:B}.

\section{Proof of Theorem C} \label{section:torelli}

Next, we consider finite-index subgroups that contain the handlebody Torelli group $\mathcal{HI}_g$, the subgroup of $\mathcal{H}_g$ that acts trivially on $H_1(\Sigma_g; \Z)$. We have $H_1(\Sigma_g; \Z) = \Z \langle a_1, \dots, a_g, b_1, \dots, b_g \rangle$, where Figure \ref{figure:homologybasisclosed} shows curves such that $[\mathfrak{a}_i] = a_i$ and $[\mathfrak{b}_i] = b_i$. 

\begin{figure}[H]
    \centering
    \includegraphics[scale=.5]{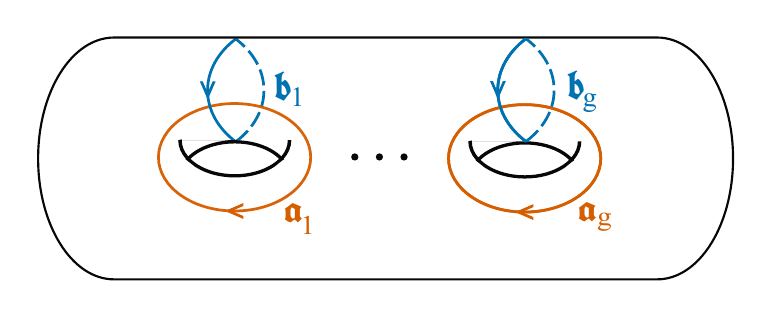}
    \caption{A basis for $H_1(\Sigma_g; \Z)$}
    \label{figure:homologybasisclosed}
\end{figure} 

\noindent The group $\Mod(\Sigma_g)$ acts on $H_1(\Sigma_g; \Z)$ by automorphisms that fix its symplectic form, and the Torelli group $\mathcal{I}_g$ is the kernel of this action. Letting $\Psi: \Mod(\Sigma_g) \rightarrow \operatorname{Sp}_{2g}(\Z)$ be the symplectic representation, this is encoded in the short exact sequence 
$$1 \rightarrow \mathcal{I}_g \rightarrow \Mod(\Sigma_g) \xrightarrow{\Psi} \operatorname{Sp}_{2g}(\Z) \rightarrow 1.$$ This action restricts to $\mathcal{H}_g$ and Hirose \cite{Hirose} proved that the image in $\operatorname{Sp}_{2g}(\Z)$ is $$\Psi(\mathcal{H}_g) \cong \left\{  \begin{bmatrix} A & 0 \\ B & (A^t)^{-1} \end{bmatrix} \; \Big| \; A \in  \operatorname{GL}_{g}(\Z), \; B = B^t \right\}.$$ Hirose chose the convention that the $\mathfrak{a}_i$ curves bound disks in the fixed handlebody, so he called this image $urSp_{2g}(\Z)$, where {\textit{ur}} stands for ``upper right." We instead have the $\mathfrak{b}_i$ curves bound disks in our handlebody, so the nontrivial entries of our matrices lie in the lower left.\footnote{We choose this convention so that $\mathcal{H}_g$ acts in the standard way on $H_1(\mathcal{V}_g; \Z) \cong \Z \langle a_1, \dots, a_g \rangle$.} The kernel of this action is $\mathcal{HI}_g := \mathcal{H}_g \cap \mathcal{I}_g$ and we have the short exact sequence $$1 \rightarrow \mathcal{HI}_g \rightarrow \mathcal{H}_g \rightarrow \Psi(\mathcal{H}_g) \rightarrow 1.$$ In this section, we will prove the following theorem. 

\begin{theorem*}[Theorem \ref{maintheorem:torelli}]
     For $g \geq 3$, let $\Gamma$ be a finite-index subgroup of $\mathcal{H}_g$ that contains the handlebody Torelli group $\mathcal{HI}_g$. Then $H_1(\Gamma; \Q) = 0$.
\end{theorem*} 

Before we prove Theorem \ref{maintheorem:torelli}, we introduce some background.

\subsection{Finite-index subgroups of $\Psi(\mathcal{H}_g)$} \label{section:finiteindexPsi} If $\Gamma$ is a finite-index subgroup of $\mathcal{H}_g$, then $\Psi(\Gamma)$ is a finite-index subgroup of $\Psi(\mathcal{H}_g)$. To prove Theorem \ref{maintheorem:torelli}, we reduce the problem of determining $H_1(\Gamma; \Q)$ to that of determining $H_1(\Psi(\Gamma); \Q)$. Thus we first have the following proposition.

\begin{proposition}
   Let $g \geq 3$ and let $\Delta$ be a finite-index subgroup of $\Psi(\mathcal{H}_g)$. Then $H_1(\Delta; \Q) = 0$. \label{proposition:urSp}
\end{proposition}

\begin{proof}
    We use the isomorphism $\Psi(\mathcal{H}_g) \cong \operatorname{GL}_g(\Z) \ltimes \operatorname{Sym}^2(\Z^g)$. Let $G := \Delta \cap \operatorname{GL}_g(\Z)$ and $S := \Delta \cap \operatorname{Sym}^2(\Z^g)$, so we have $$\Delta = G \ltimes S.$$ The groups $G$ and $S$ are of finite index in $\operatorname{GL}_g(\Z)$ and $\operatorname{Sym}^2(\Z^g)$, respectively. Then $$H_1(\Delta; \Q) = H_1(G; \Q) \times H_1(S; \Q)_G.$$ Since $\operatorname{GL}_g(\Z)$ has property (T) when $g \geq 3$, we know $H_1(G; \Q) = 0$. Since $S$ is abelian, we have $$H_1(S; \Q)_G = (S / [G,S]) \otimes \Q.$$ The group $[G,S]$ is generated by the set $$\{ g(s) - s \; | \; g \in G, \; s \in S \}.$$ The group $\operatorname{GL}_g(\Z)$ acts irreducibly on $\operatorname{Sym}^2(\Z^g)$, so given nonzero $s \in S$, the $\operatorname{GL}_g(\Z)$-orbit of $s$ is all of $\operatorname{Sym}^2(\Z^g)$. It follows that the $G$-orbit of $s$ is finite-index in $\operatorname{Sym}^2(\Z^g)$, and thus finite-index in $S$. \end{proof}

Note that $\Psi(\mathcal{H}_g)$ contains a subgroup \begin{equation} \label{equation:SLZ} \left\{ \begin{bmatrix} A & 0 \\ 0 & (A^t)^{-1} \end{bmatrix} \; \Big| \; A \in \operatorname{SL}_g(\Z) \right\} \cong \operatorname{SL}_g(\Z). \end{equation} If $\Psi(\Gamma)$ is a finite-index subgroup of $\Psi(\mathcal{H}_g)$, then the intersection of $\Psi(\Gamma)$ with the subgroup in Equation \ref{equation:SLZ} is isomorphic to a finite-index subgroup of $\operatorname{SL}_g(\Z)$. This motivates the following lemma.

\begin{lemma} \label{lemma:coinvariantsSL}
    Let $g \geq 2$. Let $\Delta$ be a finite-index subgroup of $\operatorname{SL}_g(\Z)$, and let $W$ be a nontrivial irreducible $\operatorname{SL}_g(\Q)$ representation. Then $W_{\Delta} = 0$.
\end{lemma}

\begin{proof} Since $\Delta$ is finite-index in $\operatorname{SL}_g(\Z)$, we know by \cite[Lemma 2.2]{PutmanSL} that $\Delta$ is Zariski dense in $\operatorname{SL}_g(\Q)$. We first claim that $W$ is an irreducible $\Delta$-representation. Suppose $W' \subset W$ were a nonzero proper $\Delta$-submodule. Then the subgroup of $\operatorname{SL}_g(\Q)$ preserving $W'$ would be a Zariski-closed subgroup containing $\Delta$. Since $\Delta$ is Zariski dense, this subgroup would be all of $\operatorname{SL}_g(\Q)$, contradicting the irreducibility of $W$ as an $\operatorname{SL}_g(\Q)$-module.

Now consider the coinvariants $W_{\Delta} = W / K$, where $K = \langle x-g(x) \; | \; x \in W, \; g \in \Delta \rangle$. The submodule $K$ is $\Delta$-invariant. We claim that $K \neq 0$. If $K = 0$, then $\Delta$ acts trivially on $W$. Since $W$ is irreducible as a $\Delta$-representation, this would force $W$ to be one-dimensional. Moreover, the kernel of the $\operatorname{SL}_g(\Q)$-action on $W$ is Zariski closed and contains the Zariski-dense subgroup $\Delta$, so it would have to be all of $\operatorname{SL}_g(\Q)$. Thus the entire group would act trivially on $W$, contradicting the assumption that $W$ is a nontrivial representation. Since $W$ is irreducible as a $\Delta$-module and $K$ is a nonzero $\Delta$-submodule, it follows that $K = W$ and $W_{\Delta} = 0$. \end{proof}

\subsection{Bounding pair annulus twists vanish} Let $g \geq 3$. Omori \cite[Theorem 1.2]{Omori} proved that $\mathcal{HI}_g$ is generated by \emph{bounding pair annulus twists}:\footnote{Omori actually proved the stronger fact that $\mathcal{HI}_g$ is generated by bounding pair annulus twists $T_x T_y^{-1}$ for which $x \cup y$ bounds a genus-1 subsurface and $x,y$ are not meridians.} products $T_x T_y^{-1}$ for which $x$ and $y$ are disjoint nonseparating simple closed curves on $\Sigma_g$ such that $x \cup y$ separates $\Sigma_g$ and bounds an embedded annulus in $\mathcal{V}_g$. We prove the following lemma.

\begin{lemma}[Bounding pair annulus twists vanish] \label{lemma:BPMvanish}
    For $g \geq 3$, let $\Gamma$ be a finite-index subgroup of $\mathcal{H}_g$ that contains $\mathcal{HI}_g$. Then, for all bounding pair annulus twists $T_x T_y^{-1}$, the class $[T_x T_y^{-1}] = 0$ in $H_1(\Gamma; \Q)$.
\end{lemma}

\begin{proof}
  Let $T_x T_y^{-1}$ be an annulus bounding pair twist. Then $x \cup y$ bounds a subsurface $S \cong \Sigma_h^2$ for some $2 < h < g$. Additionally, there exists a separating meridian $z$ that bounds a subsurface $T \cong \Sigma_h^1$ that contains $S$. An example is illustrated in Figure \ref{figure:BPM}.

  \begin{figure}[H]
    \centering
    \includegraphics[scale=.4]{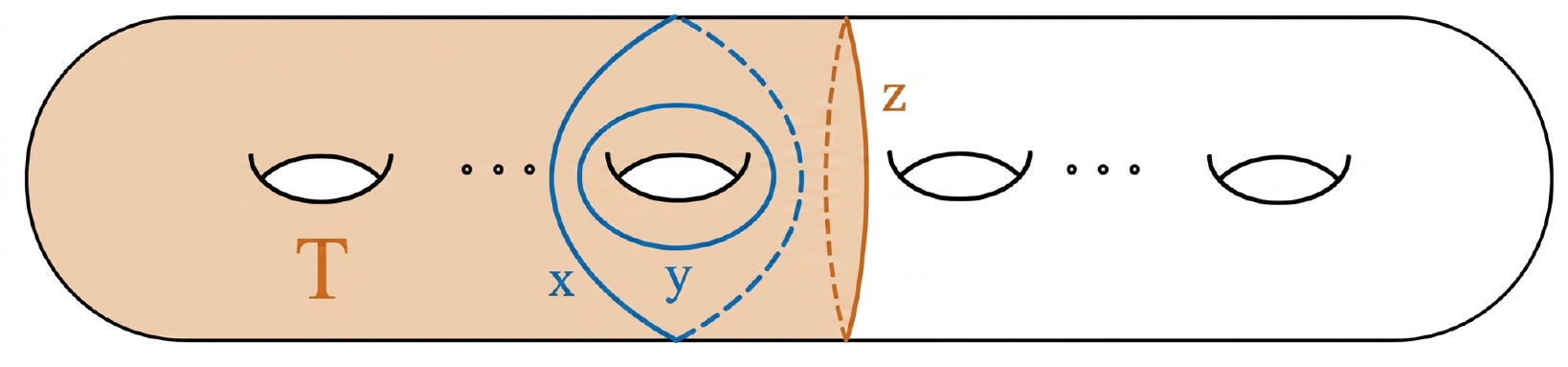}
    \caption{Bounding pair annulus twist and subsurface $T$}
    \label{figure:BPM}
\end{figure}

\noindent In particular, since $z$ bounds a disk in $\mathcal{V}_g$, the subsurface $T$ bounds a handlebody homeomorphic to $\mathcal{V}_h^1$. We have an inclusion $i: \mathcal{H}(T) \rightarrow \mathcal{H}_g$ and we let $\Gamma_T := i^{-1}(\Gamma)$. It suffices to show that $[T_x T_y^{-1}] = 0$ in $H_1(\Gamma_T; \Q)$. Note that $x$ and $y$ become homotopic in $\overline{T} \cong \Sigma_h$ after capping off the boundary with a disk, so $T_x T_y^{-1}$ dies in $\mathcal{H}(\overline{T})$. Recalling the Birman exact sequence from Section \ref{section:birman}, we have $$1 \rightarrow \pi_1(U \Sigma_h) \rightarrow \mathcal{H}(T) \rightarrow \mathcal{H}(\overline{T}) \rightarrow 1,$$ with $T_x T_y^{-1} \in \pi_1(U \Sigma_h)$. Restricting to $\Gamma_T$, we have the short exact sequence $$1 \rightarrow \pi_1(U \Sigma_h) \rightarrow \Gamma_T \rightarrow \Gamma_{\overline{T}} \rightarrow 1.$$ Here, we used the assumption that $\Gamma$ contains $\mathcal{HI}_g$ (which contains $\pi_1(U \Sigma_g)$), and so $\Gamma_T$ contains $\pi_1(U \Sigma_h)$. The associated five-term exact sequence contains the terms $$H_1(\pi_1(U \Sigma_h); \Q)_{\Gamma_{\overline{T}}} \rightarrow H_1(\Gamma_T; \Q) \rightarrow H_1(\Gamma_{\overline{T}}; \Q) \rightarrow 0.$$ It suffices to show that $H_1(\pi_1(U \Sigma_h); \Q)_{\Gamma_{\overline{T}}} = 0$. To prove this, we use another version of the Birman exact sequence which involves capping off boundaries with a \emph{marked} disk. See \cite[Chapter 4.2]{Primer}. Taking rational abelianization of the short exact sequence $$1 \rightarrow \Z \rightarrow \pi_1(U \Sigma_h) \rightarrow \pi_1(\Sigma_h) \rightarrow 1$$ gives us the exact sequence $$\Q \rightarrow H_1(\pi_1(U \Sigma_h); \Q) \rightarrow H_1(\pi_1(\Sigma_h); \Q) \rightarrow 0.$$ The map $\Q \rightarrow H_1(\pi_1(U \Sigma_h); \Q)$ is the zero map because of the boundary relation, so we get an isomorphism $$H_1(\pi_1(U \Sigma_h); \Q) \cong H_1(\pi_1(\Sigma_h); \Q).$$ Taking coinvariants with respect to the action of $\Gamma_{\overline{T}}$, we get $$H_1(\pi_1(U \Sigma_h); \Q)_{\Gamma_{\overline{T}}} \cong H_1(\pi_1(\Sigma_h); \Q)_{\Gamma_{\overline{T}}}.$$  The vector space $H_1(\pi_1(\Sigma_h); \Q) \cong \Q^{2h}$ decomposes into two nontrivial irreducible $\operatorname{SL}_h(\Q)$-representations. By Equation \ref{equation:SLZ}, the group $\Gamma_{\overline{T}}$ contains a finite-index subgroup of $\operatorname{SL}_h(\Z)$. By Lemma \ref{lemma:coinvariantsSL}, we have $H_1(\pi_1(\Sigma_h); \Q)_{\Gamma_{\overline{T}}} = 0$, and our result follows.
\end{proof}

\subsection{Proof of Theorem \ref{maintheorem:torelli}} We are now ready to prove Theorem \ref{maintheorem:torelli}.

\begin{proof}[Proof of Theorem \ref{maintheorem:torelli}] Let $f$ be in $\mathcal{HI}_g$. Since $f$ is a product of bounding pair annulus twists, the class $[f]$ vanishes in $H_1(\Gamma; \Q)$ by Lemma \ref{lemma:BPMvanish}. The five-term exact sequence of the extension $$1 \rightarrow \mathcal{HI}_g \rightarrow \Gamma \rightarrow \Psi(\Gamma) \rightarrow 1$$ contains the terms $$H_1(\mathcal{HI}_g; \Q)_{\Psi(\Gamma)} \xrightarrow{i} H_1(\Gamma; \Q) \rightarrow H_1(\Psi(\Gamma); \Q) \rightarrow 0.$$ Since all elements of $\mathcal{HI}_g$ vanish in $H_1(\Gamma; \Q)$, the map $i$ is the zero map and we have $$H_1(\Gamma; \Q) = H_1(\Psi(\Gamma); \Q).$$ Since $\Psi(\Gamma)$ is finite-index in $\Psi(\mathcal{H}_g)$, Proposition \ref{proposition:urSp} tells us that $H_1(\Psi(\Gamma); \Q) = 0$, and our result follows. \end{proof} 

\section{Proof of Theorem D}

The Johnson kernel $\mathcal{K}_g$ is the infinite-index subgroup of $\mathcal{I}_g$ generated by \emph{separating twists}: Dehn twists about separating curves. The \emph{handlebody Johnson kernel}, denoted $\mathcal{HK}_g$, is the intersection $\mathcal{H}_g \cap \mathcal{K}_g$. If $\Gamma$ is a subgroup of $\mathcal{H}_g$, define $K(\Gamma)$ to be the subgroup of $\Gamma$ generated by the set $$\{ T_{\gamma}^n \in \Gamma \; | \; \gamma \text{ is a separating meridian} \}.$$ Note that $K(\Gamma)$ is a subgroup of $\Gamma \cap \mathcal{HK}_g$. In this section, we will prove the following theorem.

\begin{theorem*}[Theorem \ref{maintheorem:johnsonkernel}]
    For $g \geq 4$, let $\Gamma$ be a finite-index subgroup of $\mathcal{H}_g$ for which $[\Gamma \cap \mathcal{HK}_g : K(\Gamma)] < \infty$. Then $H_1(\Gamma; \Q) = 0$.
\end{theorem*}

We conjecture the following.

\begin{conjecture} \label{conjecture:rationalabelHI}
    The handlebody Johnson kernel $\mathcal{HK}_g$ is generated by separating meridian twists.
\end{conjecture}

If Conjecture \ref{conjecture:rationalabelHI} is true and $\Gamma$ contains $\mathcal{HK}_g$, then $K(\Gamma) = \mathcal{HK}_g$. In this case, we get the following corollary of Theorem \ref{maintheorem:johnsonkernel}, greatly strengthening Theorem \ref{maintheorem:torelli}.

\begin{corollary} \label{corollary:johnsonkernel}
    Assume Conjecture \ref{conjecture:rationalabelHI}. For $g \geq 4$, let $\Gamma$ be a finite-index subgroup of $\mathcal{H}_g$ that contains the handlebody Johnson kernel $\mathcal{HK}_g$. Then $H_1(\Gamma; \Q) = 0$.
\end{corollary}

Before we prove Theorem \ref{maintheorem:johnsonkernel}, we introduce some background and lemmas. 

\subsection{The quotient $\mathcal{HI}_g / \mathcal{HK}_g$} Recall from Section \ref{section:torelli} that $H := H_1(\Sigma_g; \Z)$ is a free abelian group $\Z \langle a_1, \dots, a_g, b_1, \dots, b_g \rangle$. The group $\bigwedge^3 H$ contains a subgroup isomorphic to $H$ given by the injection $$x \mapsto x \wedge (\sum_{i=1}^g a_i \wedge b_i).$$
Define $V := \Z \langle a_1, \dots, a_g \rangle$ and its dual $V^* := \Z \langle b_1, \dots, b_g \rangle$. Then $H$ decomposes as $$H \cong V \oplus V^*,$$ and we have $$H_1(\mathcal{V}_g; \Z) \cong V.$$ We define the following subgroup of $(\bigwedge^3 H) / H$: $$U := [((\bigwedge^2 V) \otimes V^*) \oplus (V \otimes (\bigwedge^2 V^*)) \oplus (\bigwedge^3 V^*)] / (V \oplus V^*).$$ In particular, $U$ is the subspace of $(\bigwedge^3 H) / H$ which contains no ``triple-$a$" terms of the form $a_i \wedge a_j \wedge a_k$. We have the following proposition.

\begin{proposition}[{Morita, \cite[Lemma 2.5]{MoritaCasson}}]
    For $g \geq 3$, we have $$\mathcal{HI}_g / \mathcal{HK}_g \cong U.$$ Moreover, there is a natural $\Psi(\mathcal{H}_g)$-action of $\mathcal{H}_g$ on $U$ via its action on $H$.
\end{proposition}

\subsection{The group $U \otimes \Q$ as an $\operatorname{SL}_g(\Q)$-representation} Recall that $\mathcal{H}_g$ acts on $H = V \oplus V^*$ by the symplectic matrices $$\Psi(\mathcal{H}_g) = \left\{  \begin{bmatrix} A & 0 \\ B & (A^t)^{-1} \end{bmatrix} \; \Big| \; A \in \operatorname{GL}_g(\Z), \; B = B^t \right\}.$$ This action extends to an action of the Zariski closure of $\Psi(\mathcal{H}_g)$ on $H_1(U; \Q) \cong U \otimes \Q$. Restricting to the semisimple subgroup \begin{equation} \label{equation:SL} \left\{ \begin{bmatrix} A & 0 \\ 0 & (A^t)^{-1} \end{bmatrix} \; \Big| \; A \in \operatorname{SL}_g(\Q) \right\} \cong \operatorname{SL}_g(\Q), \end{equation} we can view $U \otimes \Q$ as an $\operatorname{SL}_g(\Q)$-representation. We have the following proposition, which can be checked using LiE \cite{LiE}.

\begin{proposition} \label{proposition:decomp}
    For $g \geq 3$, the group $U \otimes \Q$ decomposes into a direct sum of three irreducible $\operatorname{SL}_g(\Q)$ representations.\footnote{Let $\Phi_{w_1, \dots, w_{g-1}}$ be the irreducible $\operatorname{SL}_g(\Q)$-representation with highest weight $(w_1, \dots, w_{g-1})$. See \cite[Lecture 15]{FH}. Then we have $U \otimes \Q \cong \Phi_{0,1,0, \dots, 0,1} \oplus \Phi_{1,0, \dots, 0,1,0} \oplus \Phi_{0, \dots, 0,1,0,0}$.}
\end{proposition}

\subsection{Finite-index subgroups of $\mathcal{H}_g / \mathcal{HK}_g$} Define $Q_g := \mathcal{H}_g / \mathcal{HK}_g$. Since $\mathcal{HK}_g$ is in its kernel, the symplectic representation $\Psi$ factors through $Q_g$ to give us the surjective map $\overline{\Psi}: Q_g \rightarrow \Psi(\mathcal{H}_g)$.

\begin{lemma} \label{lemma:Qg}
    For $g \geq 3$, let $Q'_g$ be a finite-index subgroup of $Q_g$. Then we have $H_1(Q'_g; \Q) = 0$.
\end{lemma}

\begin{proof} Recall that we have $\mathcal{HI}_g / \mathcal{HK}_g \cong U$. Restricting the short exact sequence $$1 \rightarrow U \rightarrow Q_g \rightarrow \Psi(\mathcal{H}_g) \rightarrow 1 $$ to $Q'_g$, we get $$1 \rightarrow B \rightarrow Q'_g \rightarrow \overline{\Psi}(Q'_g) \rightarrow 1.$$ The corresponding five-term exact sequence contains the exact sequence $$H_1(B; \Q)_{\overline{\Psi}(Q'_g)} \rightarrow H_1(Q'_g; \Q) \rightarrow H_1(\overline{\Psi}(Q'_g); \Q) \rightarrow 0.$$

Since $\overline{\Psi}(Q'_g)$ is finite-index in $\Psi(\mathcal{H}_g)$, we have $$H_1(\overline{\Psi}(Q'_g); \Q) = 0$$ by Proposition \ref{proposition:urSp}. It remains to show that $H_1(B; \Q)_{\overline{\Psi}(Q'_g)} = 0$. The group $B \subset U$ is abelian and, because it is of finite index, we have\footnote{This follows from tensoring the exact sequence $0 \rightarrow B \rightarrow U \rightarrow U/B \rightarrow 0$ with $\Q$.} $$H_1(B; \Q) = H_1(U; \Q).$$
By Equation \ref{equation:SL} and Proposition \ref{proposition:decomp}, the group $H_1(B; \Q)$ decomposes as a direct sum of three irreducible $\operatorname{SL}_g(\Q)$-representations. The group $\overline{\Psi}(Q'_g)$ contains a subgroup isomorphic to a finite-index subgroup of $\operatorname{SL}_g(\Z)$ (see Section \ref{section:finiteindexPsi}), so by Lemma \ref{lemma:coinvariantsSL} we have $$H_1(B; \Q)_{\overline{\Psi}(Q'_g)} = 0$$ and our result follows.
\end{proof}

\subsection{Proof of Theorem \ref{maintheorem:johnsonkernel}} We are now ready to prove Theorem \ref{maintheorem:johnsonkernel}.

\begin{proof}[Proof of Theorem \ref{maintheorem:johnsonkernel}]

    The last terms of the five-term exact sequence of the extension $$1 \rightarrow \Gamma \cap \mathcal{HK}_g \rightarrow \Gamma \rightarrow \Gamma / (\Gamma \cap \mathcal{HK}_g) \rightarrow 1$$ are $$H_1(\Gamma \cap \mathcal{HK}_g; \Q)_{\Gamma / (\Gamma \cap \mathcal{HK}_g)} \xrightarrow{i} H_1(\Gamma; \Q) \rightarrow H_1(\Gamma / (\Gamma \cap \mathcal{HK}_g) ; \Q) \rightarrow 0.$$ Since $\Gamma / (\Gamma \cap \mathcal{HK}_g)$ is a finite-index subgroup of $Q_g$, Lemma \ref{lemma:Qg} tells us $$H_1(\Gamma / (\Gamma \cap \mathcal{HK}_g); \Q) = 0.$$ By assumption, we have $[\Gamma \cap \mathcal{HK}_g : K(\Gamma)] < \infty$. Then by Lemma \ref{lemma:finiteindexhomology}, the map $$H_1(K(\Gamma); \Q) \rightarrow H_1(\Gamma \cap \mathcal{HK}_g; \Q)$$ is surjective. The group $K(\Gamma)$ is generated by separating meridian twists, which by Theorem \ref{maintheorem:multitwistsdie} vanish in $H_1(\Gamma; \Q)$ when $g \geq 4$. Hence the map $i$ is the zero map and our result follows. \end{proof}

\appendix

\section{An alternate proof of Corollary \ref{corollary:kernelaction}} \label{appendixKernelAction}

In this section, we include an alternate proof of Corollary \ref{corollary:kernelaction} without invoking Theorem \ref{maintheorem:B}. A similar argument can be used to prove Theorem \ref{maintheorem:torelli} in an alternate way. We restate Corollary \ref{corollary:kernelaction}.

\begin{proposition}[Corollary \ref{corollary:kernelaction}]
    For $g \geq 4$, let $\Gamma$ be a finite-index subgroup of $\mathcal{H}_g$ such that the kernel of the action of $\Gamma$ on $\pi_1(\mathcal{V}_g)$ is generated by meridian multitwists. Then $H_1(\Gamma; \Q) = 0$.
\end{proposition}

First, we have the following lemma.

\begin{lemma} \label{lemma:generators}
    Consider an exact sequence of groups $$1 \rightarrow K \rightarrow G \rightarrow Q \rightarrow 1.$$ Let $\tilde{Q}$ be a set consisting of one lift $\tilde{q} \in G$ for each generator of $Q$. Then $G$ is generated by the set $K \cup \tilde{Q}$.
\end{lemma}

\begin{proof} Let $G'$ be the group generated by $K \cup \tilde{Q}$. Then the following diagram commutes and has exact rows:

\[\begin{tikzcd}[ampersand replacement=\&]
	1 \& K \& {G'} \& Q \& 1 \\
	1 \& K \& G \& Q \& 1
	\arrow[from=1-1, to=1-2]
	\arrow[from=1-2, to=1-3]
	\arrow["{=}", from=1-2, to=2-2]
	\arrow[from=1-3, to=1-4]
	\arrow[from=1-3, to=2-3]
	\arrow[from=1-4, to=1-5]
	\arrow["{=}", from=1-4, to=2-4]
	\arrow[from=2-1, to=2-2]
	\arrow[from=2-2, to=2-3]
	\arrow[from=2-3, to=2-4]
	\arrow[from=2-4, to=2-5]
\end{tikzcd}\] Our result follows from the five lemma. \end{proof}

We are now ready to prove Corollary \ref{corollary:kernelaction}.

\begin{proof}[Proof of Corollary \ref{corollary:kernelaction}]

Let $$\alpha: \mathcal{H}_g \rightarrow \operatorname{Out}(F_g)$$ be the surjective map given by the action of $\mathcal{H}_g$ on $\pi_1(\mathcal{V}_g)$. Let $\Gamma$ be a finite-index subgroup of $\mathcal{H}_g$ such that $K := \operatorname{ker}(\alpha|_{\Gamma})$ is generated by meridian multitwists. Note that $Q := \alpha(\Gamma)$ is a finite-index subgroup of $\operatorname{Out}(F_g)$. We have the short exact sequence $$1 \rightarrow K \rightarrow \Gamma \rightarrow Q \rightarrow 1$$ and, by Lemma \ref{lemma:generators}, $\Gamma$ is generated by the sets: \begin{itemize}
    \item $K$; and
    \item A set $\tilde{Q}$ consisting of one lift $\tilde{q} \in \Gamma$ for each generator of $Q$.
\end{itemize} By Theorem \ref{maintheorem:multitwistsdie}, all generators of $K$ vanish in $H_1(\Gamma; \Q)$. Since $Q$ is finite-index in $\operatorname{Out}(F_g)$, we must have $H_1(Q; \Q) = 0$. Then for any generator $\tilde{q}$ of $\tilde{Q}$, there must exist $m \in \Z$ such that $\tilde{q}^m$ can be written as a product of commutators of elements in $\tilde{Q}$. Thus we have $\tilde{q}^m \in [\Gamma, \Gamma]$ and $[\tilde{q}]=0$ in $H_1(\Gamma; \Q)$. \end{proof}

\end{document}